\documentclass[a4paper,notitlepage,twoside,leqno,12pt]{amsart}

\usepackage{bbm,pifont,latexsym}
\usepackage{anysize}
\marginsize{2.8cm}{2.8cm}{2.8cm}{2.8cm}

\usepackage{dcolumn,indentfirst}
\usepackage[colorlinks=false,linkcolor=black,citecolor=black,urlcolor=black,bookmarksopen=true,linktocpage=true,pdfstartview=FitH]{hyperref}
\usepackage{amsmath,amssymb,amscd,amsthm,amsfonts,mathrsfs}
\usepackage{color,graphicx,xcolor,graphics,subfigure,extarrows,caption2}
\usepackage{titletoc}

\newtheorem{thm}{Theorem}[section]
\newtheorem*{main lemma}{Main Lemma}

\newtheorem{lem}[thm]{Lemma}

\theoremstyle{definition}

\usepackage{enumerate,enumitem}

\makeatletter\@addtoreset{equation}{section}\makeatother

\begin{document}

\author[Y. Zhang \& G.  Zhang  ]{Yanhua Zhang \& Gaofei Zhang }
\address{
School of Mathematical Sciences, Qufu Normal University, Qufu 273165, P. R. China}

\email{g\rule[-2pt]{0.2cm}{0.5pt}zhangyh0714@163.com}
\address{Department of Mathematics, Nanjing  University, Nanjing 210093, P. R. China}
\email{zhanggf@hotmail.com}

\title[Constructing  entire functions by quasiconformal surgery]{Constructing   entire functions with non-locally connected Julia set by quasiconformal surgery }

\begin{abstract}
We give an alternative way to construct an  entire function with quasiconformal surgery so that all its Fatou components are quasi-circles but the Julia set is non-locally connected.
\end{abstract}



\date{\today}



\maketitle


\section{Statements of the main results}

In a recent manuscript \cite{BFR} the authors there construct an interesting example of   entire function for which  all the Fatou components are quasi-circles but the Julia set is non-locally connected. The idea used in their construction follows Eremenco and Sodin \cite{ES} and is called   Maclane-Vinberg method. The purpose of this article is to provide an alternative way to construct such functions by quasiconformal surgery. Our   construction is very much inspired by  \cite{Bis}. Compared with the method used in \cite{BFR}, which is purely function theoretic, ours has more geometric feature.

The starting point of our construction is a quasi-regular map $F$ which is described as follows. The building blocks of $F$ are the unit disk $\Bbb D$  and  a sequence of squares $\{P_n\}_{n\ge 1}$ with diameters being even integers.
We assemble these objects from the right to the left so that $\Bbb D$ is the rightmost one.  See Figure 1 for an illustration.     Recall that a map $f: \Bbb C \to \Bbb C$ is called $K$-quasi-regular if $f =  g \circ \phi$ for some entire function $g$ and $K$-quasiconformal map $\phi$ of the plane.

\begin{figure}[!htpb]
  \setlength{\unitlength}{1mm}
  \centering
  \includegraphics[width=100mm]{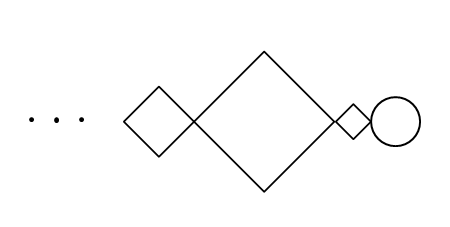}
  \caption{Constructing the quasi-regular map $F$}
  \label{Figure-1}
\end{figure}
\begin{main lemma}
There is a $K > 1$
 such that  for any $\{P_n\}_{n\ge 1}$
there is a $K$-quasi-regular branched covering $F: \Bbb C \to \Bbb C$
satisfying $F^{-1}(\Bbb D)  = \Bbb D \cup \bigcup_{n\ge 1} P_n$, and moreover,
\begin{itemize}

 \item $F: \Bbb D \to \Bbb D$ is the square map, and  $F$ is symmetric about the real axis, that is, $F(\bar{z}) = \overline{F(z)}$ for all $z \in \Bbb C$,
\item $F: P_n  \to \Bbb D$ is  a branched covering of degree $d_n$ with $d_n$ being the diameter of $P_n$, and moreover, each  $P_n$ contains  exactly one critical point  in its interior which is the center and is mapped to the origin, \item    the set of all the other critical points  consists of the  vertices of the squares   which belong to the negative real axis, which are all mapped to the point $1$,    \item    $F$ has no asymptotic values and has exactly two critical values $0$ and $1$,
    \item For any compact subset $X \subset \Bbb C$, there is a $B > 0$ depending only on $X$   such that $|F(z)| < B$ for all $z \in X$.
\end{itemize}

\end{main lemma}

As a consequence we have
\begin{thm}
There is an   entire function $f$ with no asymptotic values and with exactly two critical values $0$ and $1$   such that $f(0) = 0$ and $f(1) = 1$, and moreover,
\begin{itemize}
\item  $f$ has a critical point at $0$,
and   all the other critical points of $f$ are negative reals and are
mapped  either  to $0$  or to    $1$,
\item all the Fatou components of $f$ are quasi-disks and are eventually mapped to the super-attracting basin at the origin,
\item the Julia set of $f$ is non-locally connected.
\end{itemize}
\end{thm}

Just after the submission of the paper,  L. Rempe sent us the following comments on this work.     The entire function  constructed here is  subhyperbolic,   has one super-attracting cycle and is unbounded on the real axis, while the   one
 constructed in \cite{BFR} is hyperbolic,   has two super-attracting cycles and is bounded on the real axis.     The difference is not very significant however.    This is because the function in \cite{BFR}  can    be used to obtain the function here through a square root transformation, and on the other hand,
  the  method here can also be used to obtain the function   in \cite{BFR} through an easy adaptation.

\section{Proof of Theorem 1.1 assuming the main lemma} Let $\Bbb C$, $\Bbb C^*$, $\Bbb T$ and $\Bbb D$ denote the complex plane, the punctured plane, the unit circle and the  unit disk respectively.
Let $F$  be the quasi-regular branched covering map of the plane guaranteed by the main lemma. Let $\mu$ be the pull back of the standard complex structure by $F$. By the symmetry of  $F$ it follows that  $\mu$ is symmetric about the real axis.
Let $\psi: \Bbb C\to \Bbb C$ be the quasiconformal homeomorphism   which solves the Beltrami equation $\psi_{\bar{z}} = \mu \psi_{z}$ and fixes $-1$  and $1$.  Then  $F \circ \psi^{-1}$ is an entire function.  By the symmetry of $\mu$ we have $\psi(\bar{z}) = \overline{\psi(z)}$. Then $\psi(\Bbb T)$ is a quasi-circle  which is symmetric about the real axis.     Let $\eta$ be the Riemann isomorphism  which maps the exterior of $\Bbb T$ to the exterior of $\psi (\Bbb T)$ and fixes $\infty$ and $1$. By  the symmetry of $\psi(\Bbb T)$, $\eta(\bar{z}) = \overline{\eta(z)}$.  Note that this implies that  $\eta(-1) = -1$.    Let
$$
G =   F \circ \psi^{-1} \circ \eta.
$$
Then $G$ is holomorphic in the exterior of $\Bbb T$  and extends continuously to $\Bbb T$. Note that $G$ maps $\Bbb T$ to $\Bbb T$, by Schwarz reflection principle, $G$ can be  continued to be a meromorphic function in $\Bbb C^*$.  Let us still denote the map by $G$.  By the construction, it is not difficult to see that $G(\bar{z}) = \overline{G(z)}$,  $G(-1) = 1$ and  $G(1) = 1$.  This implies that $G$ is post-critically finite on $\Bbb T$.    Since  $\psi$  is normalized and with the quasiconformal constant being  uniformly bounded by the main lemma, $G$ is holomorphic in an annular neighborhood of $\Bbb T$ with modulus bounded away from zero.   This implies that  $G|\Bbb T$ belongs to a compact family of  analytic circle mappings.   By  Theorem 6.2 of Chapter 3 in \cite{MS},   there is a quasi-symmetric circle homeomorphism $h: \Bbb T \to \Bbb T$  such that  $G |\Bbb T= h^{-1} \circ \Lambda \circ h$   with  $\Lambda$ being  the square map.  It is not difficult to see that $h$ is also symmetric about the real axis in the sense $\overline{h(e^{it})} = h(e^{-it})$. In particular£¬ $h(1) = 1$ and $h(-1) = -1$.  Since $G|\Bbb T$ belongs to a compact family of analytic circle mappings,   the quasi-symmetric constant of $h$ is  bounded by some uniform constant $1< M<\infty$.   \begin{lem}  The quasi-symmetric circle mapping
$h$ can be extended to a quasiconformal homeomorphism  $H: \Bbb D  \to \Bbb D$  so that
 \begin{itemize}
 \item  $H(0) = 0$ and $H$ is symmetric about the real axis,   that is, $H(\bar{z}) = \overline{H(z)}$,  \item
  the quasiconformal constant  of $H$  is   bounded by some   constant depending only on $M$. \end{itemize} \end{lem}
  \begin{proof}    We first define a homeomorphism  $h: [-1, 1] \to [-1, 1]$. To do that, note that
   $G$ is symmetric about the real axis and thus  $G'(1) = \lambda > 1$ is a real number.
 For $k \ge 0$, let $x_k = 1 - \frac{1}{{\lambda}^k}$  and $y_k = 1 - \frac{1}{{2}^k}$.   Define $h(x_k) = y_k$ for all $k \ge 0$,  and  define $h$ on each $[x_k, x_{k+1}]$ so that $h$ maps $[x_k, x_{k+1}]$ linearly onto $[y_k, y_{k+1}]$.  For $-1 \le x \le 0$, define $h(x) = -h(-x)$.  In this way we have defined a homeomorphism $h: \partial U \to \partial U$ so that $h(0) = 0$ where $U$ is the upper half unit disk.
Let  $\xi: U \to \Bbb D$ be the Riemann isomorphism which  fixes $-1$ and $1$ and maps $0$ to $-i$.   Then  $\xi$ behaves like the square map near  $-1$ and $1$ because it opens the right angles to straight angles at the two points.  From the definition of $h$ on $[-1, 1]$, especially its behavior near the two end points $1$ and $-1$,  it is not difficult to check that  $\xi \circ h \circ \xi^{-1}: \Bbb T \to \Bbb T$ is quasi-symmetric with the quasi-symmetric constant depending only on that of $h$.  Hence $\xi \circ h \circ \xi^{-1}$ can be extended to a quasiconformal homeomorphism $\Pi: \Bbb D\to \Bbb D$ with the quasiconformal constant bounded by some uniform constant.  Define $H: U \to U$ by setting  $H(z) = \xi^{-1} \circ \Pi \circ \xi$. On the lower half unit disk, we define $H$ by setting $H(\bar{z}) = \overline{H(z)}$.
This completes the proof of the lemma.
  \end{proof}
Define
\begin{equation}
\widehat{G}(z) =
\begin{cases}
G(z)  & \text{for $z \in {\Bbb C} - \mathbb{D}$}, \\
H^{-1} \circ \Lambda \circ H(z)  & \text{
for $z \in \mathbb{D}$}.
\end{cases}
\end{equation}
Let $\nu_0$ be the complex structure in $\Bbb D$ given by the pull back of the standard one  by $H$. Since $H$ is symmetric about the real axis, $\nu_0$ is symmetric about the real axis.  We then pull back $\nu_0$ to the complex plane by the iteration of $G$  and get a  $\widehat{G}$-invariant complex structure $\nu$.  Thus $\nu$ is symmetric about the real axis.  Since all the maps involved  are uniformly quasiconformal, we have  $\|\nu\|_\infty < \kappa$ for some uniform $0< \kappa < 1$.  Let $\Phi$ be the quasiconformal homeomorphism of the plane which fixes $0$ and $1$ and solves the Beltrami equation given by $\nu$.     Since $\nu$ is symmetric about the real axis, so is $\Phi$.
Let $$f = \Phi \circ \widehat{G} \circ \Phi^{-1}.$$  Then $f$ is an entire function.
\begin{lem}\label{part-1} The following assertions hold.
\begin{itemize}
\item  $f$ has no asymptotic values,
\item  $f$ is symmetric about the real axis, and  fixes $0$ and $1$,
\item $f$ has a critical point at $0$ and all the other critical points   are negative reals and are mapped either to $0$ or to $1$,
\item  all the Fatou components of $f$ are quasi-disks and are eventually mapped to the super-attracting component  at the origin.
\end{itemize}
\end{lem}

\begin{proof} Suppose $f$ has an asymptotic value.  From the construction of $f$, $\widehat{G}$ and thus $F$ would have an asymptotic value. This contradicts the main lemma.  This proves the first assertion.

Since $\widehat{G}$ and $\Phi$ are both symmetric about the real axis, so is $f$. Since $\Phi$ and $\widehat{G}$ fix $0$ and $1$, $f$ fix $0$ and $1$ also. This proves the second assertion.

 By the symmetry of $\Phi$, $\Phi$ maps the real axis to the real axis. Since $\Phi$ fixes $0$ and $1$, $\Phi$ maps the negative reals to the negative reals.  From the construction of $f$, the critical set of $f$ are the $\Phi$-image of that of $\widehat{G}$.  The critical points of $\widehat{G}$, except the one at the origin, are exactly those of $G$ which belong to $(-\infty, -1]$. By the construction of $G$, these are the $(\eta^{-1} \circ \psi)$-image of the critical points of $F$ which belongs to $(-\infty, -1]$. By the symmetry and the normalization conditions of $\psi$ and $\eta$,    $\eta^{-1} \circ \psi$ maps $(-\infty, -1]$   homeomorphically to $(-\infty, -1]$.  This, together with the fact that  $\Phi$ maps $(-\infty, 0]$ homeomorphically to $(-\infty, 0]$, implies  that all the other critical points of $f$ are negative reals. Since $F$ maps any of its critical points either to $1$ or to $0$,  so does  $\widehat{G}$ by the definition of  $\widehat{G}$. Since $\Phi$ fixes $0$ and $1$,  it follows that $f$ maps any of its critical points either to $0$ or to $1$. This proves the third assertion.

 As we have seen,   $f$ has   two singularly values, $0$ and $1$, both of which are critical values.  By \cite{EL} and \cite{GK}, $f$ has neither Baker domains nor wandering domains.  Since entire functions have no Herman rings,   every
 Fatou component of $f$  must be  eventually  mapped into some  periodic cycle of either Siegel disks, or parabolic components,  or attracting components, or supper-attracting components.  Note that both the two singular values of $f$ are fixed.    Since each point of
 the boundary of a  Siegel disk must be accumulated by the forward orbit of the singular values,  $f$ has no Siegel disks.  Since any parabolic basin or attracting basin must contain an infinite forward orbit of some singular value, $f$ have not periodic parabolic or attracting cycles. Thus all the periodic cycles of Fatou components of $f$  must be super-attracting.  Let $\Omega_0$ be the Fatou component containing the origin. Since the origin is the only    periodic (fixed) critical point of $f$,    all the Fatou components must be eventually mapped to  $\Omega_0$.  Since $\widehat{G}$ is conjugate to the square map on $\overline{\Bbb D}$ and since the iteration of the square map is normal in $\Bbb D$ but not normal at any point in $\partial \Bbb D$,  we must have
    $\Omega_0 = \Phi(\Bbb D)$. In particular, $\Omega_0$ is a quasi-disk.  By the construction of $f$,   any
    component of $f^{-1}(\Omega_0)$  other than $\Omega_0$   is the $(\Phi \circ \eta^{-1} \circ \psi)$-image of some $P_n$ with $n \ge 1$,  and is thus a quasi-disk. Since the closure of each $\Phi \circ \eta^{-1} \circ \psi (P_n), n \ge 1$, does not intersect the forward orbit of the singular values, its pre-images under the iteration of $f$ must be quasi-disks also. So all the Fatou components of $f$ are quasi-disks. This proves the last assertion of the lemma.
\end{proof}

Now let us  prove a uniform contraction property of $G$. Let ${\rm diam}(\cdot)$ denote the diameter with respect to the Euclidean metric.
\begin{lem}\label{uniform-contraction}
For any $L > 1$ there is an $n \ge 1$ depending only on $L$  and  a curve segment $\Gamma$ attached to $\Bbb T$ from the outside such that \begin{itemize} \item  ${\rm diam}(\Gamma)< 1$ and \item   $G^n:  \Gamma \to [-L, -1]$ is a homeomorphism. \end{itemize}
\end{lem}
\begin{proof} Let $h$ be the quasi-symmetric circle homeomorphism which conjugates $G|\Bbb T$ to the square map.
For $k \ge 1$ let $z_k =  h^{-1}(e^{i\pi/2^k})$.  Since  the quasi-symmetric constant of $h$ is bounded by some uniform $1< M < \infty$, it follows that  $z_k \to 1$ uniformly in the following sense: for any $\epsilon > 0$, there is an $N \ge 1$ independent of $G$ so that  $|z_k -1| < \epsilon$ for all $k \ge N$.  Let $\Gamma_k$ be the curve segment attached  to  $z_k$ from the outside of $\Bbb T$ so that $G^k : \Gamma_k \to [-L, -1]$ is a homeomorphism.
Let $X =  \Bbb C  \setminus \{0,  1\}$ and $Y =  G^{-1}(X)$. Then $G:  Y \to  X$ is a holomorphic covering map. From $G(-1) = 1$ it follows that  $Y \subset X \setminus \{-1\}$.  Thus the pull back by $G$
   decreases the hyperbolic metric in $X$.   Let $l_X(\cdot)$ denote the length with respect to the hyperbolic metric in $X$. Then $l_X(\Gamma_k) \le  l_X([-L, -1])$ for all $k \ge 1$.    This, together with that  $z_k \to 1$ uniformly, implies the existence of the integer $n$ such that  ${\rm diam}(\Gamma_n) < 1$.
\end{proof}

To construct the entire function $f$ with non-locally connected Julia sets,  we need to
 take appropriate squares $\{P_n\}_{n\ge 1}$ in the construction of the quasi-regular map $F$.  As we have seen above,  $\Bbb D$   corresponds to a quasi-disk  $\Omega_0$  which is fixed by  $f$, and  each  $P_n$   corresponds to a quasi-disk $\Omega_n$   which is mapped to $\Omega_0$ by $f$.   More precisely,
 $$
 \partial \Omega_n = \Phi \circ \eta^{-1} \circ \psi (\partial P_n) \:\:\:\: \hbox{  and  }  \partial \Omega_0 = \Phi  (\Bbb T).
 $$By symmetry, all $\Omega_n$  are symmetric about the real axis.  From the construction, it also follows that $\{-1, 1\} \subset \partial \Omega_0$ and $f(-1) = f(1) = 1$.  Let $\partial \Omega_l \cap \Bbb R = \{a_{l}, b_l\}$ with $a_l > b_l$.    Since  $\Phi$ is  uniformly quasiconformal and normalized, it follows that there exists some $R > 0$ independent of the choice of $\{P_n\}_{n\ge 1}$ such that $$\overline{\Omega_0} \subset B_R(0)$$ where $B_R(0)$ is the disk centered at the origin and with radius $R$. From the fact that $\psi(-1) = \eta(-1) = \Phi(-1) = -1$ it follows that $a_1 = -1$.   Since $\Phi$, $\eta$ and $\psi$ are  uniformly quasiconformal and normalized, we can take  $P_1$      large enough  so that $|b_1| > R +1$ holds
 for any choice of the remaining squares $\{P_n\}_{n\ge 2}$.     Let $n_1 = 0$.   Now for $k \ge 1$, suppose the diameters of all $P_l$, $1 \le l \le k$,  have been determined such that  for each $1 \le l \le k$,    there is an integer $n_l$ and a branch of $f^{-n_1}$, say $\chi_l$,   such that $|\chi_l(a_l)| < R$ and $|\chi_l(b_l)| > R+1$ hold  for any  choice of $\{P_n\}_{n\ge l+1}$.   Let us determine  the diameter of $P_{k+1}$ and $n_{k+1}$ as follows.    Since $P_k$ has determined and since $\Phi$, $\eta$ and $\psi$ are uniformly quasiconformal and normalized,    it follows that $a_{k+1} = b_k$   is bounded and the bound does not depend on the choice  of $\{P_n\}_{n\ge k+1}$.    By Lemma~\ref{uniform-contraction}  and the fact that $\Phi$ is uniformly quasiconformal and normalized, we can choose $n_{k+1}$  independent of the choice of $\{P_n\}_{n\ge k+1}$,   such that the corresponding  branch of $f^{-n_{k+1}}$  , say $\chi_{k+1}$,    satisfying  $|\chi_{k+1}(a_{k+1})|< R$.  Since  $F$ is uniformly bounded on compact sets  by the main lemma, and since  $\Phi$, $\eta$ and $\psi$  are uniformly quasiconformal and normalized, it follows that $f$ is uniformly bounded on compact sets. Thus  for the $n_{k+1}$, we can choose $P_{k+1}$ large enough  to ensure that   $|\chi_{k+1}(b_{k+1})| > R+1$ holds for any choice of $\{P_n\}_{n\ge k+2}$.

 In this way we get a sequence of squares $\{P_n\}$.  For such $\{P_n\}$, let $F$ be the quasi-regular map in the main lemma and  $f$ be the   entire function obtained as above.
 \begin{lem}\label{imme-1} There  exist    an $\epsilon > 0$  and an infinite   sequence of Fatou components of $f$ such that  the spherical diameter of each Fatou component in this sequence is greater than $\epsilon$.
 \end{lem} \begin{proof}
 From the construction of  $P_n$,  it follows that there exist a constant $R > 0$ and an infinite   sequence of Fatou components of $f$ which intersect both of the two Euclidean circles $\{z: \:|z| = R\}$ and $\{z: \: |z| = R +1\}$.   It is clear that the spherical diameter of these Fatou components have a uniform positive lower bound depending only on $R$.
 \end{proof}

Before we show that $J(f)$ is not locally connected let us  recall a criterion of the local connectedness.

\begin{thm}[{Whyburn, \cite{Why42}}]\label{Whyburn's-criterior}
A closed set in $\widehat{\Bbb C}$  is locally connected if and only if it satisfies the following two conditions:
\begin{enumerate}
\item  The boundary of each complementary component of it is locally connected;
\item  For an arbitrary $\epsilon>0$, the number of the complementary components whose diameters with respect to the spherical metric exceed $\epsilon$ is finite.
\end{enumerate}
\end{thm}

\begin{lem}\label{part-2}
$J(f)$ is not locally connected.
\end{lem}

\begin{proof}

By Lemma~\ref{part-1} all the Fatou components of $f$ are bounded and simply connected.  By Theorem 2 in \cite{Ki}   $J(f)$ is connected. Thus   $J(f) \cup \{\infty\}$ is a continuum in $\widehat{\Bbb C}$. By Corollary 5.13  in \cite{Na} a continuum in $\widehat{\Bbb C}$  can not fail to be locally connected at a single point. It follows that $J(f)$ is locally connected if and only if $J(f) \cup \{\infty\}$ is locally connected.
By Lemma~\ref{imme-1}    and  Theorem~\ref{Whyburn's-criterior} it follows that $J(f) \cup  \{\infty\}$  is not locally connected. Thus $J(f)$ is not locally connected.  This proves Lemma~\ref{part-2}.
\end{proof}

Theorem 1.1 now follows from Lemmas~\ref{part-1} and ~\ref{part-2}.

\section{Proof of the main lemma}

  Suppose  $\{P_n\}$  is a sequence of squares in the Main Lemma.  See Figure 2 for an illustration.
Here the diameters of $P_1$, $P_2$ , $P_3$ and $P_4$  are $2$, $4$, $2$ and $8$ respectiely. We divide the plane  into disjoint union of
``half strips"  and squares.  Up to symmetry, there are five types of half strips.  Except the three special ones whose shapes are independent of the choice of $\{P_n\}_{n\ge 1}$ (see Figures 3-5),  each of the other ones has the following common properties.
\begin{itemize}
\item  The two bottom angles are $\pi/4$ and $3 \pi /4$,
 \item  the width is  equal to one,
   \item the length of the bottom is equal to  $\sqrt 2$,
   \item   it is  bent to a right angle at some place.
   \end{itemize}

   \begin{figure}[!htpb]
  \setlength{\unitlength}{1mm}
  \centering
  \includegraphics[width=100mm]{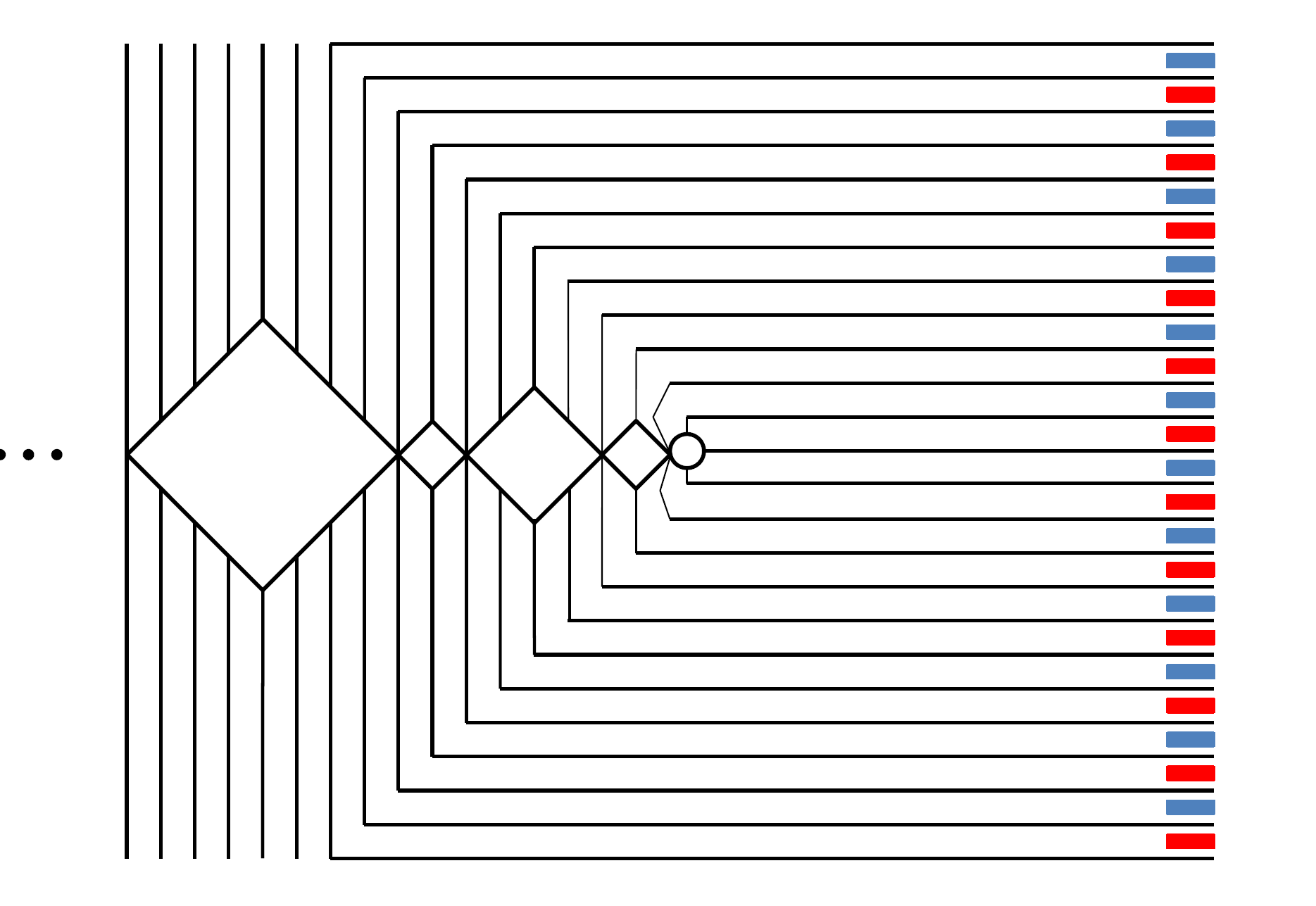}
  \caption{Constructing the quasi-regular map $F$}
  \label{Figure-2}
\end{figure}

We will construct $F$ by defining it  to be  $K$-quasiconformal on each of these half strips and squares   with   $K> 1$ being some constant
independent of the choice of $\{P_n\}$.  We then   glue these maps together along the boundary of their domains so that the global map satisfies the properties in the lemma.  First we will define a quasiconformal map $g$ on each of the half strip so that it maps each half strip   onto  the standard half strip
$S =  \{x + iy\:|\: -1 < x <1, y > 0\}$.   In this process the situations for types II and III are  similar with the type I and the situation for type V is similar with the type IV.  So we will only present the details of the construction
for the half strips of types I and IV.  It should be easy for the reader to supply the details for the other situations.

Let  $H$  be the half strip of the type I.   Let
 $AB$ denote the $1/4$-unit circle and $A'B'$ denote the bottom side of $S$. See Figure 3 for an illustration. Define $g$ on $AB$ so that it maps $AB$ proportionally onto $A'B'$.  Let $L$ and $R$ be the two components of $\partial H \setminus AB$ so that $L$ is the one with $B$ being the  end point and $R$  is  the one with $A$ being the end point.  Similarly, let $L'$ and $R'$ be the two corresponding  components of $\partial S \setminus A'B'$.
 For any $X \in L$, define $g(X) = X' \in L'$ such that the subarc $BX \subset L$ has the same length as the subarc $B'X' \subset L'$, and for $Y \in R$, define $g(Y) = Y' \in R'$ such that the subarc $AY \subset R$ has the same length as the subarc $A'Y'\subset R'$. In this way we have defined $g: \partial H \to \partial S$ so that it preserves the length on the two boundary sides $L$ and $R$. To define $g$ in the interior of $H$, let $C$ be the vertex of $H$ at the upper left corner and take $D \in L$ and $E \in R$ such that $|CD| = 2$ and $|AE|= 1$. Let $C' = g(C)$, $D' = g(D)$ and $E' = g(E)$. Define $g$ on $DE$ so that it maps $DE$ linearly onto $D'E'$.
  Now it is not difficult to extend  $g$ to  be  a quasiconformal homeomorphism from the domain $ABCDE$ to the pentagon $A'B'C'D'E'$.  For instance, one can  divide the two  domains into the same number of  star domains and then apply Theorem 6 of \cite{NS}.    Let $H_0$  and $S_0$  denote the remaining part of $H$  and $S$  respectively. Note that $g$ has been defined on $\partial H_0$ and is linear on each straight boundary piece of $H_0$.
   We can thus  extend $g$ to be a linear homeomorphism from $H_0$ to $S_0$.   It is clear the map $g: H \to S$ obtained in this way is a quasiconformal homeomorphism, and most importantly, it preserve the length on the two boundary sides of $H$.  The similar construction can be applied to the half strips of types II and III. See Figures 4 and 5 for an illustration. We leave the details to the reader.
\begin{figure}[!htpb]
  \setlength{\unitlength}{1mm}
  \centering
  \includegraphics[width=80mm]{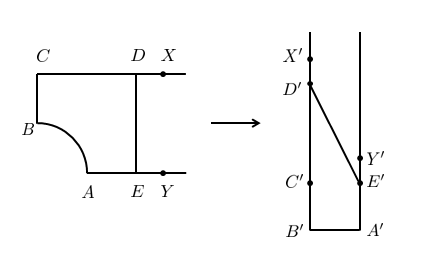}
  \caption{Define $g$ in the half strip of type I}
  \label{Figure-3}
\end{figure}

\begin{figure}[!htpb]
  \setlength{\unitlength}{1mm}
  \centering
  \includegraphics[width=80mm]{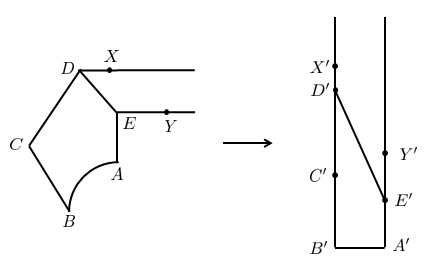}
  \caption{Defining $g$ in the  half strip of type II}
  \label{Figure-4}
\end{figure}

\begin{figure}[!htpb]
  \setlength{\unitlength}{1mm}
  \centering
  \includegraphics[width=80mm]{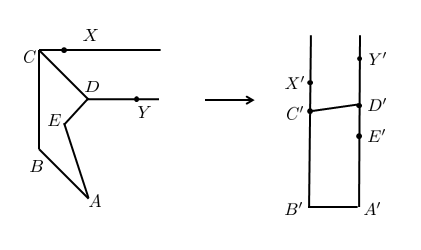}
  \caption{Define $g$ in the half strip of type III}
  \label{Figure-5}
\end{figure}

  Now let us define $g$ in a half strip $H$  of type IV. See Figure 6 for an illustration.   As before we first define $g$ on the bottom side $AB$ of $H$ to be linear on   $AB$ so that $g$ maps $AB$ to   $A'B'$.      Then  let $L$  and $R$ denote respectively the two components of $\partial H \setminus AB$,  one with $B$  being the end point    and    the other one with $A$ being the end point.    Similarly, Let $L'$ and $R'$ denote the corresponding components of $\partial S \setminus A'B'$.   For $X \in  L$, define $g(X) = X' \in L'$  so that the subarc $BX \subset L$  has the same length as the subarc $B'X' \subset L'$.  Similarly, for $Y \in R$, define $g(Y) = Y' \in R'$ so that the subarc $AY \subset R $  has the same length as the subarc $A'Y' \subset R'$.   In this way we have defined $g: \partial H \to \partial S$.

Let $D$ and $E$  denote the two corner points in $\partial H$ as indicated in Figure 6.  To define $g$ in the interior of $H$, we take a points  $C \in L$  such that $|BC| = |AE|$.   Let $C', D',  E' \in \partial S$ so that $C' = g(C)$, $D' = g(D)$ and $E' = g(E)$. Then $|CD| = |C'D'|$, $|BC| = |B'C'|$, $|AE| = |A'E'|$  and $g$ maps $CD$, $BC$ and $AE$
 linearly onto $C'D'$, $B'C'$ and $A'E'$, respectively. Define $g$ on $CE$ and $DE$ so that $g$ maps them linearly onto $C'E'$ and $D'E'$ respectively. Then $g$ can be extended linearly   to the interiors of the parallelogram $ABCE$ and the triangle $CED$ so that $g$ maps them onto the interiors of the  rectangle $A'B'C'E'$ and the triangle $C'E'D'$ respectively. The remaining part of $H$, say $H_0$,  is a half strip with $DE$ being the bottom side.  Correspondingly, the remaining part of $S$, say $S_0$, is also a half strip with $D'E'$ being the bottom side. Note that $g$ has been defined on $\partial H_0$ so that $g: \partial H_0 \to \partial S_0$ is linear on each straight boundary piece. Thus $g$ can be extended to the interior of $H_0$ so that $g$ maps $H_0$ linearly onto $S_0$.       Now we define $g: H \to S$  by gluing the maps on the three pieces along the boundary. It is clear that $g$ is a quasiconformal map. Since the quasiconformal constant of $g$ on each piece is some uniform constant, thus  the quasiconformal constant of $g$ on $H$ is a uniform constant.  Moreover, from the construction we see that $g$ preserve the length on the two boundary side of $\partial H$.

\begin{figure}[!htpb]
  \setlength{\unitlength}{1mm}
  \centering
  \includegraphics[width=80mm]{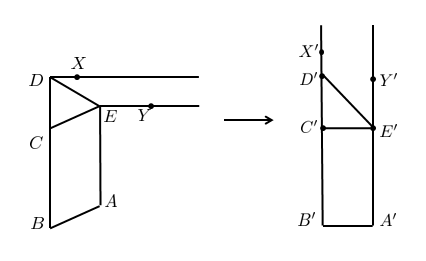}
  \caption{Define $g$ in the half strip of type IV}
  \label{Figure-6}
\end{figure}
\begin{figure}[!htpb]
  \setlength{\unitlength}{1mm}
  \centering
  \includegraphics[width=80mm]{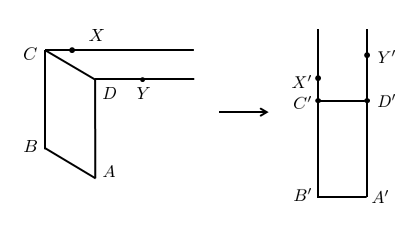}
  \caption{Define $g$ in the half strip of type V}
  \label{Figure-7}
\end{figure}
For a half strip $H$ of type V,  the method above can be applied in a similar way to get  a $K$-quasiconformal homeomorphism  $g: H \to S$  such that   (1)   $K >1$ is some uniform constant,  and (2)    $g$ is linear on the bottom side and  preserves the length on the two boundary sides of $H$. See Figure 7 for an illustration. We leave the details to the reader.

Now let us define $F$  on the half strips. Let $H_{+}$ and $H_{-}$  denote respectively the upper and lower components of
$\Bbb C  \setminus (-\infty, -1] \cup \overline{\Bbb D} \cup [1, +\infty)$.   Define $\sigma: S \to H_+$ by $\sigma(z) = e^{i\pi(1-z)/2}$ and $\tau: S \to H_-$ by $\tau(z) = -e^{i\pi(1-z)/2}$.  It  remains to verify that whenever a point belong to the intersection of the boundaries of two half strips, the two maps defined on the two half strips must coincide on this point. To see this, we may assume that $z$ belong to the $L$ part of the boundary of some red half strip (that is, marked with red rectangle) and belongs to the $R$ part of the boundary of some blue half strip (that is, marked with blue rectangle). Let $g_r$, $F_r$, and $g_b$, $F_b$  denote respectively the maps $g$ and $F$  defined on the red and blue half strips. Since both $g_r$ and $g_b$ preserve the length on the two boundary side of the half strips, $g_r(z)\in L' $ and $g_b(z)\in R' $ have the same imaginary part. By the definition of $\sigma$, it follows that $$ F_r(z) = \sigma (g_r(z)) = -\sigma(g_b(z)) = \tau (g_b(z)) = F_b(z).$$

\begin{figure}[!htpb]
  \setlength{\unitlength}{1mm}
  \centering
  \includegraphics[width=100mm]{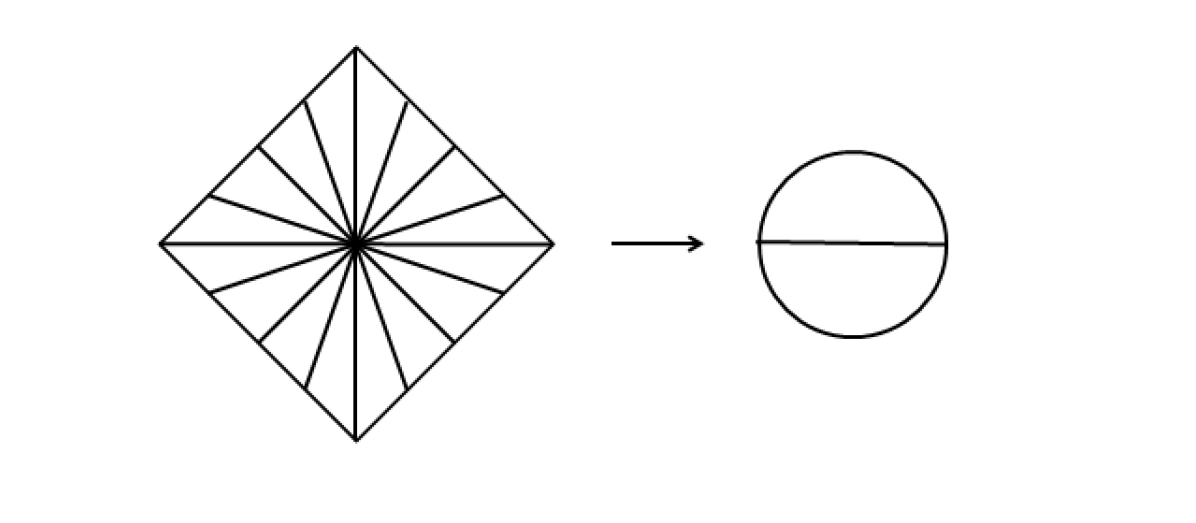}
  \caption{$\theta \to \xi(\theta)$, $r \to \frac{r^d}{R(\theta)^d}e^{\xi(\theta)}$  }
  \label{Figure-7}
\end{figure}

It remains to define $F$ in the interior of the squares. Suppose that  the diameter of $P_n$ is  $d =  2m$ for some integer $m \ge 1$. Note that we
have defined $F$ on the boundary of $P_n$ in the above. From the construction, $F$ maps $\partial P_n$ proportionally onto $\Bbb T$ and   $F:  \partial P_n \to \Bbb T$ is a covering of degree $d$.  Suppose $\omega$ is the center of $P_n$. For $0\le \theta \le 2 \pi$, let $R(\theta) > 0$  and $0\le \xi(\theta) \le 2 \pi$  be the numbers such that $\omega + R(\theta) e^{i \theta} \in \partial P_n$    and
$$
F(\omega + R(\theta) e^{i \theta}) =   e^{i\xi(\theta)}.
$$   Since $F$ maps $\partial P_n$ proportionally onto $\Bbb T$, it follows that $\xi'(\theta)$ exists except the fours angles corresponding to the four vertices  of $P_n$, and moreover,  $1/C  <\xi'(\theta_1)/\xi'(\theta_2) < C$ for some uniform $C > 1$. Since $F: \partial P_n \to \Bbb T$ is a covering of degree $d$ we get  $\xi'(\theta)/d > 0$ and is uniformly bounded away from zero and $1$.
For $ 0 \le r  \le R(\theta)$, define
$$F(\omega + r e^{i \theta}) =  \frac{r^d}{R(\theta)^d}  e^{i\xi(\theta)} .$$
We  claim that  $F$ is $K$-quasi-regular for some universal $K > 1$ in $P_n$.   For the convenience of the computation we adopt the log coordinate. Consider the map $$\zeta \to \log F (\omega +e^\zeta) = u + iv$$ where $\zeta = x + iy$.  Then $x = \log r$ and $y = \theta$.  We thus get   $$u(x, y)= dx   - d\log R(y) \hbox{  and  } v (x, y) = \xi(y).$$
Let $h = u + iv$. Let us compute $h_{\bar{z}}$ and $h_{z}$.
 For $A, B >0$ we denote $A=  O^*(B)$  if $A/B$ is  uniformly bounded away from zero and $1$.    We use $O(1)$ to denote a number whose absolute value is bounded by some universal constant.  By a simple computation we have
  $$
  h_x = d \hbox{  and  } h_y = -d \frac{R'(y)}{R(y)} + i \xi'(y).
  $$   Note that  $R'(y) = O(1)$,  $R(y) = O^*(d)$ and $\xi'(y) = O^*( d)$. From $h_z = \frac{1}{2}(h_x - ih_y)$ and $h_{\bar{z}} = \frac{1}{2}(h_x + ih_y)$ we get
$$
h_{\bar{z}} = \frac{1}{2}(d - O^*(d)) + i\lambda(z)
$$
and
$$
h_{z} = \frac{1}{2}(d + O^*(d)) - i\lambda(z)
$$ where $\lambda(z) =  O(1)$.  This implies that $|h_{\bar{z}}/h_z| <k < 1$  for some  uniform   $0< k < 1$  independent of $d$.

Now from the construction
 the first four assertions of the main lemma hold.  To prove the last one,  we  first take a smallest closed disk centered at the origin which contains  the compact set.  Note that the bound of $F$ on the closed disk
is essentially determined by the length of the longest half strip contained in it.   But this  is  then  bounded above by some constant depending only on the diameter of the disk.  Since the diameter of the disk depends on the compact set, the last assertion follows.  The proof of the Main Lemma is completed.

\vspace{0.5cm}

$\bold{Acknowledgements.}$  We would like to thank D. Sixsmith and L. Rempe for their detailed comments on an early version of this work. Further thanks are due to the anonymous referees for their   suggestions which greatly improve the paper.  The second author is partially supported by NSFC(11171144, 11325104).


\begin{thebibliography}{99}



\bibitem{BFR}Bergweiler W, Fagella N, Rempe-Gillen L. Hyperbolic entire functions with bounded Fatou components. Comment Math Helv(4), 2015, 90: 799--829
\bibitem{Bis}  Bishop C.   Constructing entire functions by quasiconformal folding. Acta Math(1), 2015, 214: 1--60

\bibitem{EL}   Eremenko A, Lyubich M.  Dynamical properties of some classes of entire functions. Ann Inst Fourier(4),  1992, 42: 989--1020

\bibitem{ES}  Eremenko A, Sodin M. Parametrization of entire functions of sinetype by their critical values, Entire and subharmonic functions. Adv Soviet Math, vol. 11, Amer Math Soc, Providence, RI, 1992, 237--242

\bibitem{GK} Goldberg L, Keen L. A finiteness theorem for a dynamical class of entire functions. Ergd Th \& Dyna Sys(2),1986, 6:183--192

\bibitem{Ki} Kisaka M. On the connectivity of Julia sets of transcendental entire functions. Ergd Th \& Dyna Sys(1), 1998,18:189--205

\bibitem{MS}  De Melo W,  Strien S. One-Dimensional Dynamics. Springer-Verlag Berlin Heidelberg, 1993

\bibitem{Na} Sam B. Nadler, Jr. Continuum theory. An introduction. Monograph and Textbooks in Pure and Applied Mathematics, vol.158, Marcel Dekker  Inc., New York, 1992

\bibitem{NS}  Nicks D A, Sixsmith D J. Periodic domains of quasiregular maps. Mathematics(4), 2017, 141: 1385--1392

\bibitem{Why42} Whyburn G T. Analytic Topology. American Mathematical Society Colloquium Publications, 28, American Mathematical Society, New York, 1942

\end{thebibliography}
\end{document}